\documentclass [a4paper,10pt]{article}
\usepackage{amssymb,amsmath,amscd,url,amsthm,latexsym}
  \usepackage{graphics} 
  \usepackage{epsfig} 

\newtheorem{theorem}{Theorem}[section]
\newtheorem{corollary}{Corollary}
\newtheorem{lemma}[theorem]{Lemma}

\theoremstyle{definition}
\newtheorem{definition}[theorem]{Definition}
\newtheorem{remark}{Remark}

\title{A bifurcation for a generalized Burger's equation in dimension one} 
\author{Jean-Fran\c cois Rault}
\date{}

\begin{document}
\maketitle

\begin{center}
LMPA Joseph Liouville, ULCO, FR 2956 CNRS, \\
Universit\'e Lille Nord de France \\
50 rue F. Buisson, B.P. 699, F-62228 Calais Cedex (France)\\
\url{jfrault@lmpa.univ-littoral.fr}\\
\end{center}
\bigskip

\begin{abstract}
We consider the generalized Burger's equation
\begin{eqnarray*}
\left\{
    \begin{array}{ll}
			\partial_t u = \partial_x^2u - u \partial_x u + u^p - \lambda u  &\textrm{ in } \overline{\Omega} \times (0,T), \\
			\mathcal{B}(u)=0  & \textrm{ on } \partial \Omega \times (0,T), \\
			u(\cdot,0) = \varphi \geq  0  &  \textrm{ in } \overline{\Omega},
    \end{array}
\right.		
\end{eqnarray*}
with $p>1$, $\lambda \in \mathbb{R}$, $T\in (0,\infty]$, $\Omega$ a subdomain of $\mathbb{R}$, and where $\mathcal{B}(u)=0$ designates some boundary conditions. First, using some phase plane arguments, we study the existence of stationary solutions under Dirichlet or Neumann boundary conditions and prove a bifurcation depending on the parameter $\lambda$. Then, we compare positive solutions of the parabolic equation with appropriate stationary solutions to prove that global existence can occur when $\mathcal{B}(u)=0$ stands for the Dirichlet, the Neumann or the dissipative dynamical boundary conditions $\sigma \partial_t u + \partial _\nu u=0$. Finally, for many boundary conditions, global existence and blow up phenomena for solutions of the nonlinear parabolic problem in an unbounded domain $\Omega$ are investigated by using some standard super-solutions and some weighted $L^1-$norms.\\

\end{abstract}
\bigskip

\section{Introduction}
Let $\Omega$ be a domain of the real line $\mathbb{R}$, not necessarily bounded. Let $p$ be a real number with $p>1$, $\lambda \in \mathbb{R}$ and $\varphi$ a non-negative continuous function in $\overline{\Omega}$. Consider the following nonlinear parabolic problem 
\begin{equation}\label{B_param}
\left\{
    \begin{array}{ll}
        \partial _t u = \partial_x^2 u -u\partial_x u + u^p - \lambda u& \textrm{ in }\overline{\Omega} \times (0,\infty), \\
       \mathcal{B}(u)= 0 & \textrm{ on } \partial \Omega \times (0,\infty), \\
        u(\cdot , 0) = \varphi & \textrm{ in } \overline{\Omega} ,
    \end{array}
\right.
\end{equation}
where $\mathcal{B}(u)= 0$ stands for the Dirichlet boundary conditions ($u=0$), the Neumann boundary conditions ($\partial_\nu u=0$) or the dynamical boundary conditions ($\sigma \partial _t u + \partial _\nu u=0$ with $\sigma$ a non-negative smooth function). In the first section, we study the stationary equation 
\begin{equation}\label{Bsat_param}
u''-uu'+u|u|^{p-1}-\lambda u =0
\end{equation}
stemming from Problem \ref{B_param}. We aim to prove the existence of positive and sign-changing solutions using phase plane arguments and dealing with the first order system
\begin{eqnarray}\label{Bstat_sys}
\left( \begin{array}{c}
u'  \\
v'
\end{array} \right)   
=
\left( \begin{array}{c}
v\\
uv - u|u|^{p-1}+\lambda u \end{array} \right) .
\end{eqnarray}
We prove a bifurcation in the phase plane of this system, depending on the parameters $\lambda$ and $p$, which influences the resolution of Equation \ref{Bsat_param} under the Dirichlet, Neumann and mixed boundary conditions. Then in a second section, using the comparison method from \cite{vBDC}, we deduce from the solutions of the stationary Equation \ref{Bsat_param} some regular super-solutions for the Problem \ref{B_param}. Dealing with these super-solutions and with the blow-up results from \cite{vBPM}, we investigate global existence and blow-up phenomena for the Problem \ref{B_param} for different values of $\lambda$ and $p$, and for the Dirichlet, the Neumann and the dynamical boundary conditions. We also examine both phenomena in unbounded domains: we obtain global existence results with the comparison method and using some well-known super-solutions (we mean explicit functions) for the Dirichlet, the Neumann and the dynamical boundary conditions. The blowing-up concerns the regular solutions of Problem \ref{B_param} satisfying some growth order at infinity and some boundary conditions such that
\begin{itemize}
\item $\partial_\nu u=0$ (Neumann b.c.),
\item $\partial_\nu u=g(u)$ with $g$ a polynomial of degree $2$ (nonlinear b.c.).
\end{itemize}
We use some weighted $L^1-$norms: our technique is to prove the blowing-up of the solution by proving the blowing-up of appropriate $L^1-$norms.\\

Before starting, let us define the kind of solution we look for:

\begin{definition}
A function $u$ is called a solution (or regular solution) of Equation \ref{Bsat_param} in $\Omega$ if $u$ is of class $\mathcal{C}^2(\Omega)$ and satisfies the equation in the classical sense.\\
A function $u$ is called a solution (or regular solution) of Problem \ref{B_param} in $\Omega$ if $u$ is of class $\mathcal{C}(\overline{\Omega} \times [0,T)) \cap \mathcal{C}^{2,1}(\overline{\Omega} \times (0,T))$, where denotes $T$ its maximal existence time, and satisfies the equations in the classical sense.
\end{definition}

\section{Stationary equation}
In this section, we study the existence of positive and sign-changing solutions of Equation \ref{Bsat_param} using a phase plane method. For the theory of phase planes (nature of equilibrium, regularity, behaviour and uniqueness of trajectory), we refer to H.Amann's book \cite{Am}. Unless otherwise stated, we suppose $p \in (1,\infty)$. First, we can note that System \ref{Bstat_sys} has three equilibrium points if $\lambda>0$: $(0,0)$, $(\lambda^\frac{1}{p-1},0)$ and $(-\lambda^\frac{1}{p-1},0)$. Using Hartman-Grobman's linearization theorem (see Reference \cite{Am}), we can state that $(0,0)$ is a saddle point, $(\lambda^\frac{1}{p-1},0)$ is an unstable and repulsive vortex (if $1-4(p-1)\lambda^\frac{p-3}{p-1}<0)$, a node (if $1-4(p-1)\lambda^\frac{p-3}{p-1}\geq 0$, which degenerates when $1-4(p-1)\lambda^\frac{p-3}{p-1} =0$ ). And $(-\lambda^\frac{1}{p-1},0)$ is a stable and attractive  vortex (for $1-4(p-1)\lambda^\frac{p-3}{p-1}<0$), a node (for $1-4(p-1)\lambda^\frac{p-3}{p-1} \geq 0$ with degeneracy when equality occurs). If $\lambda \leq 0$, then $(0,0)$ is the only equilibrium point of System \ref{Bstat_sys}. We will prove later that $(0,0)$ is a center.

\subsection{Case $\lambda >0$}  
Let $\lambda$ be a positive real number and $p>1$. We want to study the phase plane of the System \ref{Bstat_sys}. First we prove a lemma on the symmetry of the trajectories:

\begin{lemma}\label{sym_lem}
The support of the trajectories of the System \ref{Bstat_sys} are symmetric with respect to the  ordinates axis.
\end{lemma}
\begin{proof}
Let $(u,v)$ denote a solution of the System \ref{Bstat_sys} in $(-a,a)$ for some $a \in (0,\infty]$, and define
\begin{eqnarray*}
\left\{
    \begin{array}{lll}
        w(x) &=& -u(-x) \\
        z(x) &=& v(-x)
    \end{array} \textrm{ for all } x\in (-a,a) .
\right.
\end{eqnarray*}
A simple calculus of the derivatives implies
\begin{displaymath}
w'(x) = u'(-x) = v(-x) = z(x),
\end{displaymath}
and
\begin{eqnarray*}
        z'(x) &=& -u(-x) \\
        		  &=& -v'(-x) \\
        			&=& - \Big[ u(-x)v(-x) -u(-x)|u(-x)|^{p-1} + \lambda u(-x) \Big] \\
        		  &=& w(x)z(x) -  w(x)|w(x)|^{p-1} + \lambda w(x).
\end{eqnarray*}
Then $(w,z)$ is also a trajectory of the System \ref{Bstat_sys}, and it is symmetric to $(u,v)$ with respect to the  ordinates axis. 
\end{proof}

Thus, we can reduce our phase plane analysis to the half plane $\mathbb{R}^+ \times \mathbb{R}$. In order to draw the phase plane of the System \ref{Bstat_sys}, we write the ordinate $v$ as a function depending on the abscissa $u$: $v=f(u)$. We do not know the function $f$, but we can deduce its variations and convexity using the equations \ref{Bstat_sys}. For the variations, we have
\begin{equation}\label{variation}
\frac{dv}{du} = \frac{uv - u|u|^{p-1}+\lambda u}{v} =\frac{u}{v} \Big( v- |u|^{p-1}+\lambda \Big),
\end{equation}
in particular, it vanishes along the axis $\{u=0\}$ and along the curve $\{v = |u|^{p-1}-\lambda\}$. For $u < \lambda^\frac{1}{p-1}$, we have
\begin{displaymath}
\frac{dv}{du} \Bigg|_{v=0}= \infty
\end{displaymath}
whereas for $u > \lambda^\frac{1}{p-1}$
\begin{displaymath}
\frac{dv}{du} \Bigg|_{v=0}= -\infty .
\end{displaymath}
Then we have $\frac{dv}{du} >0$ in the sets $\{ u>0 , v>0 , v>|u|^{p-1}-\lambda\}$ and  $\{ u>0 , v<0 , v<|u|^{p-1}-\lambda\}$. On the other hand, $\frac{dv}{du}<0$ in the sets $\{ u>0 , v<0 , v>|u|^{p-1}-\lambda\}$ and  $\{ u>0 , v>0 , v<|u|^{p-1}-\lambda\}$. Next, we compute the convexity of the function $f$ and we obtain
\begin{equation}\label{convexity}
\frac{d^2v}{du^2} = 1+ \frac{1}{v^2} \Bigg[ (\lambda -pu^{p-1})v - u(\lambda -u^{p-1}) \frac{dv}{du}\Bigg].
\end{equation}
We have $\frac{d^2v}{du^2}<0$ in $\{ u>0, v>0, v <|u|^{p-1} - \lambda \}$ and $\frac{d^2v}{du^2}>0$ in $\{ u>0, v<0, v <|u|^{p-1} - \lambda \}$. Since
\begin{displaymath}
\frac{d^2v}{du^2} \Bigg|_{u=0}= 1+ \frac{\lambda}{v} \textrm{ and } \frac{d^2v}{du^2} \Bigg|_{v=|u|^{p-1} - \lambda}= (1-p) \frac{|u|^{p-1}}{v} ,
\end{displaymath}
the convexity is sign-changing in $\{ u>0, v> |u|^{p-1} - \lambda\}$. These arguments are sufficient to know the profile of the trajectories in the half plane $\{ v < |u|^{p-1} - \lambda \}$.  We do not need to know how the trajectories behave in $\{ u>0, v<0, v >|u|^{p-1} - \lambda \}$ to solve Equation \ref{Bsat_param}. In $\{ u>0, v>0, v >|u|^{p-1} - \lambda \}$, things are different: unbounded trajectories can appear (see \S \ref{unbounded}). To ensure the occurrence of bounded trajectories, we need an additional hypothesis:
\begin{equation}\label{p_3}
p \geq 3 .
\end{equation}

\begin{lemma}\label{bounded_lem}
Under hypothesis \ref{p_3}, all the trajectories of the System \ref{Bstat_sys} are bounded in $A=\{ u>0, v>0, v >|u|^{p-1} - \lambda \}$.
\end{lemma}
\begin{proof}
Let $v_0>0$ and consider $(u,v)$ the solution of the System \ref{Bstat_sys} with initial data $(u(0),v(0))=(0,v_0)$. The calculus of the variations (see Equation \ref{variation}) ensures that $(u(t),v(t)) \in A$ for small $t>0$. We prove that there exist $0<\tau<\infty$ such that $v(\tau)= |u(\tau)|^{p-1} - \lambda$. It means that $(u,v)$ is bounded in $A$. Since $(u,v)$ belongs to $A$, we have
\begin{displaymath}
\frac{dv}{du} = u + \frac{\lambda u}{v} - \frac{u|u|^{p-1}}{v} \leq u + \frac{\lambda u}{v}.
\end{displaymath}
Then $\frac{dv}{du} \geq 0$ in $A$ implies $v>v_0$ as long as $(u,v) \in A$, and we obtain
\begin{displaymath}
\frac{dv}{du}  \leq u \Big(1+ \frac{\lambda }{v_0} \Big) .
\end{displaymath}
Integration gives
\begin{displaymath}
v \leq \frac{1}{2}\Big(1+ \frac{\lambda }{v_0} \Big) u^2 + v_0 .
\end{displaymath}
If $p>3$, the intersection $\{ v=|u|^{p-1} - \lambda \} \cap\{ v= \frac{1}{2}\Big(1+ \frac{\lambda }{v_0} \Big) u^2 + v_0 \}$ is non-empty for all $v_0>0$. If $p=3$,  we need to choose $v_0$ sufficiently big such that 
\begin{displaymath}
\frac{1}{2}\Big(1+ \frac{\lambda }{v_0} \Big) <1 .
\end{displaymath}
Then, the trajectory $(u,v)$ belongs to the compact 
\begin{displaymath}
\{ u\geq 0, v \geq |u|^{p-1} - \lambda, v \leq \frac{1}{2}\Big(1+ \frac{\lambda }{v_0} \Big) u^2 + v_0\},
\end{displaymath}
and, using $ \frac{dv}{du} \geq 0$, we know that there exist $0<\tau<\infty$ such that $v(\tau)= |u(\tau)|^{p-1} - \lambda$. This argument proves that each solution of the System \ref{Bstat_sys} with initial data $(u(0),v(0))=(0,v_0)$ is bounded in $A$ if $v_0$ is big enough. Thanks to uniqueness of solution, it also proves the result for all the solutions initiated in $A$.
\end{proof}

\begin{figure}[h]
\begin{center}
  \includegraphics[width=4in]{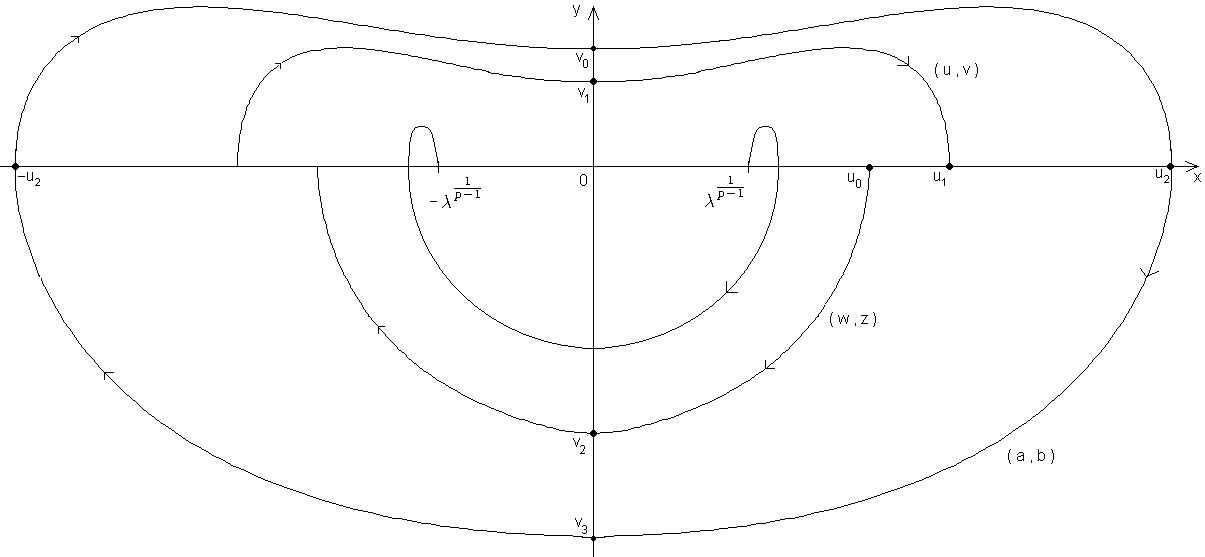}\\
  \caption{Phase plane for $p\geq 3$ and $\lambda>0$.}\label{phase1}
  \end{center}
\end{figure}

Then, we complete this phase plane analysis by proving the existence of periodic trajectories.

\begin{lemma}\label{lem-per}
Assume that hypothesis \ref{p_3} is fulfilled. Then, there exists periodic trajectories of the System \ref{Bstat_sys}.
\end{lemma}
\begin{proof}
Thanks to the symmetry (see Lemma \ref{sym_lem}), we just need to prove that for some initial data belonging to $\{0\} \times (0,\infty)$, there exists a trajectory which attains a point belonging to $\{0\} \times (-\infty,0)$. First, consider a trajectory $(u,v)$ initiated at $(0,v_1)$ with $v_1>0$. According to hypothesis \ref{p_3}, we know that $(u,v)$ is bounded, and using its variations and its convexity (Equations \ref{variation} and \ref{convexity}), we can deduce that $(u,v)$ attains the $x-$axis at a point $(u_1,0)$ with $u_1 > \lambda^\frac{1}{p-1}$ (see Figure \ref{phase1}). Then, using the reverse system
\begin{eqnarray*}
\left( \begin{array}{c}
u'  \\
v'
\end{array} \right)   
=
\left( \begin{array}{c}
-v\\
-uv + u|u|^{p-1}-\lambda u \end{array} \right) ,
\end{eqnarray*}
and one of its trajectories initiated at $(0,v_2)$ with $v_2<-\lambda)$ (trajectories of reverse system and of System \ref{Bstat_sys} have same support) , one can note that for $u_0>\lambda^\frac{1}{p-1}$, there exist a trajectory $(w,z)$ with $w(0)=u_0$ and $z(0)=0$ (see Figure \ref{phase1}). Finally, let us consider the trajectory $(a,b)$ of System \ref{Bstat_sys} containing the point $(u_2,0)$. Thanks to the uniqueness of the solutions, and using the information on the variations and the convexity, we deduce that there exist two real numbers $s<t$ such that $a(s)= a(t) =0$, $b(s)= v_0$ and $b(t) =v_3$ (see Figure \ref{phase1}). Thus, the trajectory $(a,b)$ is the periodic trajectory we look for.
\end{proof}

Now, analysing the phase plane of the System \ref{Bstat_sys}, we deduce the following results concerning the equation \ref{B_param}.

\begin{theorem}\label{sol+} 
Assume hypothesis \ref{p_3} and $\lambda>0$. For some $\alpha>0$ and for each boundary conditions
\begin{itemize}
\item $u(-\alpha)=u(\alpha)=0$ (Dirichlet b.c.) ,
\item $u'(-\alpha)=u'(\alpha)=0$ (Neumann b.c.) ,
\item $u(-\alpha)=u'(\alpha)=0$ (mixed$-1$ b.c.), 
\item $u'(-\alpha)=u(\alpha)=0$ (mixed$-2$ b.c.), 
\end{itemize}
there exists a unique positive solution of the Equation \ref{B_param}
\begin{displaymath}
u'' -uu'+u|u|^{p-1}-\lambda u = 0 \textrm{ in } (-\alpha,\alpha) .
\end{displaymath}
\end{theorem}
\begin{proof}
We use the phase plane of System \ref{Bstat_sys}, see Figure \ref{phase1}. Consider the trajectory $(a,b)$ between the points 
\begin{itemize}
\item $(0,v_0)$ and $(0,v_3)$: we obtain the Dirichlet solution,
\item $(0,v_0)$ and $(u_2,0)$: we obtain the mixed$-1$ solution,
\item $(u_2,0)$ and $(0,v_3)$: we obtain the mixed$-2$ solution.
\end{itemize}
For the Neumann solution, consider $0<\mu_0<\lambda^\frac{1}{p-1}$ and the trajectory $(\mu,\nu)$ of System \ref{Bstat_sys} initiated at $(\mu_0,0)$. Since $(\mu,\nu)$ can not cross the trajectory $(u,v)$ ( see Figure \ref{phase1}), it must cross the $x-$axis at $(\mu_1,0)$ with $\lambda^\frac{1}{p-1}<\mu_1<u_1$. Thus, the abscissa of this trajectory is the Neumann solution we look for. Uniqueness of solution comes from standard ODE's theorems applied to the System \ref{Bstat_sys}. Finally, the length ($2\alpha$) of the existence interval is governed by the time needed by the trajectory to go from its initial data to its ``final data''.
\end{proof}

\begin{theorem}
Assume hypothesis \ref{p_3} and $\lambda>0$. For some $\alpha>0$, there exists a periodic sign-changing solution of the Equation \ref{B_param}
\begin{displaymath}
u'' -uu'+u|u|^{p-1}-\lambda u = 0 \textrm{ in } \mathbb{R} .
\end{displaymath}
\end{theorem}
\begin{proof}
We just need to choose one the periodic trajectories of the System \ref{Bstat_sys} built in Lemma \ref{lem-per}.
\end{proof}

\begin{remark}
Using the periodic solutions in the previous theorem, we can build four sign-changing solutions satisfying the four boundary conditions: Dirichlet, Neumann, mixed$-1$ and mixed$-2$ (see Theorem \ref{sol+}). 
\end{remark}

Now, suppose that hypothesis \ref{p_3} is not achieved. Then, we do not know if the solutions are bounded in $\{v>|u|^{p-1} -\lambda \}$: we will see in \S\ref{unbounded} that unbounded solutions appear. But in $\{v<|u|^{p-1} -\lambda \}$, the behaviour of the trajectories do not change.

\begin{theorem}\label{p<3_sol+}
Let $\lambda>0$. For some $\alpha>0$, there exists a unique positive solution of the Equation \ref{B_param}
\begin{displaymath}
u'' -uu'+u|u|^{p-1}-\lambda u = 0 \textrm{ in } (-\alpha,\alpha) 
\end{displaymath}
with the mixed boundary conditions $u'(-\alpha)=u(\alpha)=0$. In addition, if 
\begin{equation}\label{lam-vortex}
1-4(p-1)\lambda^\frac{p-3}{p-1}<0 ,
\end{equation}
then there exists a unique positive solution of the equation \ref{B_param} under the Neumann boundary conditions.
\end{theorem}
\begin{proof}
The first part of the statement comes from Theorem \ref{sol+}, the solution with mixed$-2$ boundary conditions is located in $\{v<|u|^{p-1} -\lambda \}$. The other part stems from Equation \ref{lam-vortex}: in this case, the equilibrium $(\lambda^\frac{1}{p-1},0)$ is an unstable vortex. If we consider $u_0>0$ such that $|\lambda^\frac{1}{p-1}-u_0|$ is sufficiently small, the trajectory $(u,v)$ of the System \ref{Bstat_sys}, with $u(0)=u_0$ and $v(0)=0$, whirls around $(\lambda^\frac{1}{p-1},0)$. Thus, there exists $\tau>0$ such that $v(\tau)=0$ and $u(t)>0$ for all $t\in [0,\tau]$.
\end{proof}

Without hypothesis \ref{p_3}, we can not construct positive solutions anymore for the Dirichlet, Neumann or mixed$-1$ boundary conditions. If we do not impose the positivity, we obtain this result:

\begin{theorem}
Let $\lambda>0$. For some $\alpha>0$ and for each boundary conditions
\begin{itemize}
\item $u'(-\alpha)=u'(\alpha)=0$ (Neumann b.c.) ,
\item $u(-\alpha)=u'(\alpha)=0$ (mixed$-1$ b.c.), 
\end{itemize}
there exists a solution of the Equation \ref{B_param}
\begin{displaymath}
u'' -uu'+u|u|^{p-1}-\lambda u = 0 \textrm{ in } (-\alpha,\alpha) .
\end{displaymath}
\end{theorem}
\begin{proof}
As we mentioned before, we consider the part $\{v<|u|^{p-1} -\lambda \}$ of the phase plane of the System \ref{Bstat_sys} (see Figure \ref{phase1}). For the Neumann solution, we consider the trajectory $(a,b)$ between $(u_2,0)$ and $(-u_2,0)$. For the mixed$-1$ solution, we can also consider the trajectory $(a,b)$, but only between $(0,v_3)$ and $(-u_2,0)$.
\end{proof}

\begin{remark}
The Neumann solution built above is sign changing, whereas the mixed$-1$ solution is negative.
\end{remark}

\begin{remark}
In the general case, we can not build any solution with the Dirichlet boundary conditions using our phase plane method. Indeed, we will give a criterion in Theorem \ref{thm_nonborne2} concerning nonexistence of the Dirichlet solution.
\end{remark}

Concerning the solutions in infinite interval, we can state:

\begin{theorem}
Let $\lambda>0$. Then the Equation \ref{B_param}
\begin{displaymath}
u'' -uu'+u|u|^{p-1}-\lambda u = 0  
\end{displaymath}
admits
\begin{itemize}
\item a positive solution $u$ in $(-\infty,0]$ satisfying $u'(-\infty)=u'(0)=0$ (Neumann).
\item a positive solution $v$ in $(-\infty,0]$ satisfying $v'(-\infty)=v(0)=0$ (mixed$-2$). 
\item a sign-changing solution $w$ in $\mathbb{R}$ satisfying $w'(-\infty)=w'(\infty)=0$ (Neumann).
\item a negative solution $u$ in $[0,\infty)$ satisfying $z(0)=z'(\infty)=0$ (mixed$-1$).
\end{itemize}
\end{theorem}
\begin{proof}
Consider $\mu_0>0$ with $\mu_0>\lambda^\frac{1}{p-1}$ and with $|\lambda^\frac{1}{p-1}-u_0|$ small enough such that there exists a trajectory $(\mu,\nu)$ of the System \ref{Bstat_sys} satisfying
\begin{displaymath}
\mu(-\infty) =  \lambda^\frac{1}{p-1} , \ \nu(-\infty)=0 \ \textrm{ and } \ \mu(0) =  \mu_0 , \ \nu(0)=0
\end{displaymath}
Hence, $u=\mu$ in $(-\infty,0]$ is suitable for the first statement. Then, the trajectory $(\mu,\nu)$ can be continued in the part $\{ u>0, v<0\}$ using the information on its behaviour (see Equations \ref{variation} and \ref{convexity}) until $(\mu,\nu)$ attains the ordinate axis. Denote $t_1>0$ the time such that $\mu(t_1)=0$ and $\nu(t_1)<0$. We obtain the second statement setting $v(t)=\mu(t+t_1)$ for all $t\in(-\infty,0]$. Finally, these results and the symmetry of the trajectories (see Lemma \ref{sym_lem}) imply the third and the fourth statements with the following definitions:
\begin{eqnarray*}
w(t) =\left\{
    \begin{array}{lll}
        v(t) & \forall & t \leq 0 \\
        -v(-t) & \forall & t>0
    \end{array} \textrm{ and } 
\right.
z(t) = -v(-t) \textrm{ for all } t\geq0 .
\end{eqnarray*}
\end{proof}

\subsection{Case $\lambda \leq 0$}
First note that the System \ref{Bstat_sys} has only one equilibrium point $(0,0)$. As in the previous case, we can reduce our phase plane analysis to the half-plane $\mathbb{R}^+\times \mathbb{R}$ since of Lemma \ref{sym_lem}. Again, we obtain some information on the variations of the trajectories of the System \ref{Bstat_sys} using Equation \ref{variation}. We have $\frac{dv}{du} =0$ along the curves $\{u=0\}$ and $\{v=|u|^{p-1}-\lambda\}$. For $u>0$
\begin{displaymath}
\frac{dv}{du} \Bigg|_{v=0}= - \infty
\end{displaymath}
whereas for $u <0$
\begin{displaymath}
\frac{dv}{du} \Bigg|_{v=0}= +\infty .
\end{displaymath}
Then, we have $\frac{dv}{du} \geq 0$ in $\{ u>0, v<0\} \cup \{ v\geq|u|^{p-1}-\lambda \} $ and $\frac{dv}{du} \leq 0$ in $\{ u>0, v>0, v\leq |u|^{p-1}-\lambda\}$. In addition, thanks to Equation \ref{convexity}, we know that $\frac{d^2v}{du^2} \leq 0$ in $\{ u>0, v>0, v\leq|u|^{p-1}-\lambda \} $, $\frac{dv}{du} \geq 0$ in $\{ u>0, v<0\}$ while it is sign-changing in $\{u>0, v\geq|u|^{p-1}-\lambda \} $. In this last part of the plane, we use the following lemma, similar to Lemma \ref{bounded_lem}:

\begin{lemma}\label{bounded_neg}
Let $\lambda \leq 0$ and $(u,v)$ be a trajectory of the System \ref{Bstat_sys} with initial data $(0,v_0)$. If $v_0>-\lambda$ satisfies
\begin{equation}\label{v0max}
\left\{
   						 \begin{array}{ll}
        				v_0 >-\lambda & \textrm{ if } p \geq 3 \ ,\\
        				v_0 \leq -\lambda +(p-1)^\frac{p-1}{3-p} -\frac{1}{2}(p-1)^\frac{2}{3-p} & \textrm{ if } p< 3 \ ,
    						\end{array}
			  \right.
\end{equation}
then the trajectory $(u,v)$ is bounded in $A=\{u>0, v\geq |u|^{p-1}-\lambda \}$.
\end{lemma}
\begin{proof}
The calculus of the variations (see Equation \ref{variation}) ensures that $(u(t),v(t)) \in A$ for small $t>0$. We prove that there exist $0<\tau<\infty$ such that $v(\tau)= |u(\tau)|^{p-1} - \lambda$. It means that $(u,v)$ is bounded in $A$. Since $(u,v)$ belongs to $A$ and thanks to $\lambda \leq 0$, we have
\begin{displaymath}
0 \leq \frac{dv}{du} = u + \frac{\lambda u}{v} - \frac{u|u|^{p-1}}{v} \leq u .
\end{displaymath}
Then, integration between $0$ and $u$ gives
\begin{displaymath}
v \leq \frac{1}{2} u^2 + v_0 .
\end{displaymath}
Hypothesis \ref{v0max} implies that $\{u>0,v= |u|^{p-1}-\lambda \} \cap \{u>0,v= \frac{1}{2} u^2 + v_0 \}$ is not empty. Thus, the trajectory $(u,v)$ belongs to the compact 
\begin{displaymath}
\{ u\geq 0, v \geq |u|^{p-1} - \lambda, v \leq \frac{1}{2} u^2 + v_0 \}.
\end{displaymath}
Using $ \frac{dv}{du} \geq 0$, we know that there exist $\tau>0$ such that $v(\tau)= |u(\tau)|^{p-1} - \lambda$.
\end{proof}

Now, the phase plane of the System \ref{Bstat_sys} can be drawn, see Figure \ref{phase2}.

\begin{figure}[h]
\begin{center}
  \includegraphics[width=3in]{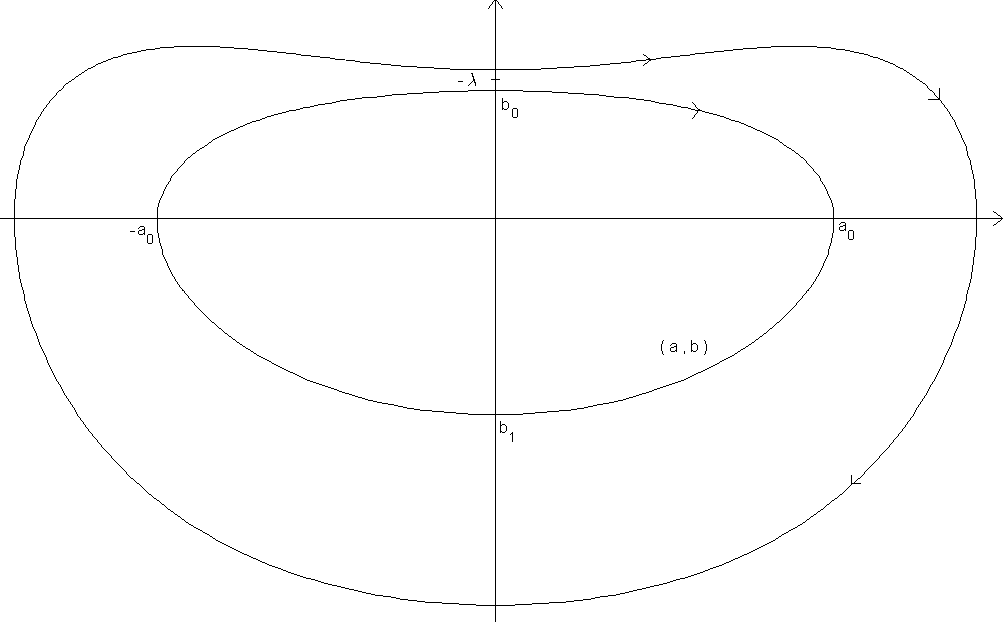}\\
  \caption{Phase plane for  $\lambda \leq 0$.}\label{phase2}
  \end{center}
\end{figure}

\begin{corollary} 
The equilibrium point $(0,0)$ is a center for the System \ref{Bstat_sys}.
\end{corollary}

Now, we use these information on the trajectories of the System \ref{Bstat_sys} to obtain some results concerning the solutions of Equation \ref{B_param}.

\begin{theorem}\label{lneg_sol+}
Let $\lambda \leq 0$. For some $\alpha>0$ and for each boundary conditions
\begin{itemize}
\item $u(-\alpha)=u(\alpha)=0$ (Dirichlet b.c.) ,
\item $u(-\alpha)=u'(\alpha)=0$ (mixed$-1$ b.c.), 
\item $u'(-\alpha)=u(\alpha)=0$ (mixed$-2$ b.c.), 
\end{itemize}
there exists a unique positive solution of the Equation \ref{B_param}
\begin{displaymath}
u'' -uu'+u|u|^{p-1}-\lambda u = 0 \textrm{ in } (-\alpha,\alpha) .
\end{displaymath}
\end{theorem}
\begin{proof}
We use the phase plane of System \ref{Bstat_sys}, see Figure \ref{phase2}. Consider the trajectory $(a,b)$ between the points 
\begin{itemize}
\item $(0,b_0)$ and $(0,b_1)$: we obtain the Dirichlet solution,
\item $(0,b_0)$ and $(a_0,0)$: we obtain the mixed$-1$ solution, 
\item $(a_0,0)$ and $(0,b_1)$: we obtain the mixed$-2$ solution.
\end{itemize}
\end{proof}

\begin{theorem} 
Let $\lambda \leq 0$. For all $\alpha>0$, the Equation \ref{B_param}
\begin{displaymath}
u'' -uu'+u|u|^{p-1}-\lambda u = 0 \textrm{ in } (-\alpha,\alpha) 
\end{displaymath}
admits no positive solution under the Neumann boundary conditions.
\end{theorem}
\begin{proof}
Ab absurbo, suppose that there exists $u$ a positive solution of \ref{B_param} under the Neumann boundary conditions, and denote $v=u'$. Then the curve $(u,v)$ is a trajectory of the System \ref{Bstat_sys} located in $\mathbb{R}^+ \times \mathbb{R}$ with initial data on the axis $\{ v=0\}$. Then  Equations \ref{variation} and \ref{convexity} prove that $(u,v)$ can not cross the axis $\{ v=0\}$ once again without going into $\mathbb{R}^- \times \mathbb{R}$. A contradiction with the positivity of $u$.
\end{proof}

\begin{theorem} 
Let $\lambda \leq 0$. For some $\alpha>0$, the Equation \ref{B_param}
\begin{displaymath}
u'' -uu'+u|u|^{p-1}-\lambda u = 0 \textrm{ in } (-\alpha,\alpha) 
\end{displaymath}
admits a sign changing solution under the Neumann boundary conditions.
\end{theorem}
\begin{proof}
Using the phase plane of System \ref{Bstat_sys} (see Figure \ref{phase2}), consider the trajectory $(a,b)$ between the points $(a_0,0)$ and $(-a_0,0)$.
\end{proof}

To conclude this section, let us give this result concerning the periodic solutions:

\begin{theorem}
Let $\lambda \leq 0$. For some $\alpha>0$, there exists a sign-changing periodic solution of the Equation \ref{B_param}
\begin{displaymath}
u'' -uu'+u|u|^{p-1}-\lambda u = 0 \textrm{ in } \mathbb{R} .
\end{displaymath}
\end{theorem}
\begin{proof}
As in Lemma \ref{lem-per}, we can build periodic trajectories of \ref{Bstat_sys} using the symmetry (Lemma \ref{sym_lem}). 
\end{proof}

\subsection{Unbounded solutions}\label{unbounded} 
In the above paragraphs, we proved that all the trajectories of the System \ref{Bstat_sys} are bounded for $p\geq 3$, but if $1<p<3$ we do not have a general answer: for example, we obtain some bounded trajectories when $\lambda \leq 0$ (see Lemma \ref{bounded_neg}), but with our method, we do not have (yet) any result when $\lambda > 0$. In this paragraph, we show that there exists unbounded trajectories for every $\lambda \in \mathbb{R}$ and for all $p \in (1,3)$. We start with a trajectory $(u,v)$ with an initial data $(0,v_0)$. 

\begin{lemma}\label{lem_nonborne}
Let $p \in (1,3)$ and $\lambda \in \mathbb{R}$. Suppose that
\begin{equation}\label{H_v0}
v_0> 2\max \{ - \lambda, 0\} + 2\cdot(8)^\frac{p-1}{3-p}.
\end{equation}
Then the trajectory $(u,v)$ is not bounded.
\end{lemma}
\begin{proof}
We will show that under hypothesis \ref{H_v0}, the trajectory $(u,v)$ always lies above the curve $\Big\{ v = 2u^{p-1} + 2\max \{ - \lambda, 0\} \Big\}$. Thus, using $\frac{dv}{du} \geq 0$ (Equation \ref{variation}),  we obtain that $(u,v)$ is not bounded. Ab absurdo, suppose that there exists $x_*>0$ such that $u(x_*)=u_1>0$ and $v(x_*)=v_1>0$ satisfy
\begin{equation}\label{non-borne2}
v_1= 2u_1^{p-1} + 2\max \{ - \lambda, 0\},
\end{equation}
and  
\begin{displaymath}
v(x)> 2u(x)^{p-1} + 2\max \{ - \lambda, 0\} \ \forall \ x \in [0,x_*).
\end{displaymath}
Thus in $[0,x_*)$, we have
\begin{equation}\label{non-borne}
\frac{\lambda - u^{p-1}}{v}>-\frac{1}{2}.
\end{equation}
On the other hand, Equation \ref{variation} gives
\begin{displaymath}
\frac{d v}{d u}= u +u\frac{\lambda -u^{p-1}}{v},
\end{displaymath}
and thanks to condition \ref{non-borne}, we obtain
\begin{equation}\label{non-borne5}
\frac{d v}{d u} \geq \frac{1}{2}u \geq0.
\end{equation}
Then $v(u) \geq \frac{u^2}{4} +v_0$. Hence, for $u=u_1$, we have:
\begin{displaymath}
v_1=v(u_1) \geq \frac{u_1^2}{4} +v_0 ,
\end{displaymath}
and by definition \ref{non-borne2} of $u_1$, we have
\begin{displaymath}
2u_1^{p-1} + 2\max \{ - \lambda, 0\} \geq \frac{u_1^2}{4} +v_0 .
\end{displaymath}
Hypothesis \ref{H_v0} implies
\begin{equation}\label{non-borne3}
 -2\cdot(8)^\frac{p-1}{3-p} > \frac{u_1^2}{4} -2u_1^{p-1}  .
\end{equation}
Meanwhile, if we study both cases $u_1<8^\frac{1}{3-p}$ and $u_1>8^\frac{1}{3-p}$, we remark that
\begin{equation}\label{non-borne4}
\frac{u_1^2}{4} -2u_1^{p-1} = \frac{u_1^{p-1}}{4} \Big( u_1^{3-p} -8  \Big) \geq -2\cdot(8)^\frac{p-1}{3-p}.
\end{equation}
Equations \ref{non-borne3} and \ref{non-borne4} are not compatible. Thus,  the trajectory $(u,v)$ can not attain the curve $\Big\{ v = 2u^{p-1} + 2\max \{ - \lambda, 0\} \Big\}$.
\end{proof}

Concerning Equation \ref{B_param}, we obtain the following results:

\begin{theorem}\label{thm_nonborne}
Let $p \in (2,3)$ and $\lambda \in \mathbb{R}$. For some $\alpha>0$, there exists a positive and unbounded solution of the Equation \ref{B_param}
\begin{displaymath}
u'' -uu'+u|u|^{p-1}-\lambda u = 0 \textrm{ in } (-\alpha,\alpha) .
\end{displaymath}
satisfying 
\begin{displaymath}
u(-a)=0 \textrm{ and } \lim_{ x\to a} u(x) = \infty.
\end{displaymath}
\end{theorem}
\begin{proof}
The existence comes from the previous lemma. We just need to prove that the length of the existence interval is finite. Ab absurdo, suppose that there exists a positive and unbounded solution $u$ of the Equation \ref{B_param} in $[0,\infty)$. Let $b>0$ such that $u>2|\lambda|^\frac{1}{p-1}$ in $[b, \infty)$, and define $w(x,t)=u(x+t)$ for all $x \in [b,b+1]$ and for all $t\in[0,\infty)$. Thanks to the choice of $b$, we have
\begin{displaymath}
\partial_x^2 u-u\partial_x u+u^p - \lambda u \geq \partial_x^2 u-u\partial_x u+\frac{u^p}{2} 
\end{displaymath}
in $[b,b+1] \times [0,\infty)$. Because the solution $u$ corresponds to a trajectory of the System \ref{Bstat_sys} located in $\mathbb{R} \times \mathbb{R}^+$, we have $\partial_t w >0$. Thus, $w$ is super-solution of the following problem
\begin{equation*}
\left\{
    \begin{array}{ll}
        \partial _t v = \partial_x^2 u -u\partial_x u + \frac{1}{2}u^p - \lambda u& \textrm{ in }[b,b+1] \times (0,\infty), \\
      \partial_t v + \partial_\nu v= 0 & \textrm{ on } \{ \pm b  \} \times (0,\infty), \\
        v(\cdot , 0) = 2|\lambda|^\frac{1}{p-1} & \textrm{ in } [b,b+1] .
    \end{array}
\right.
\end{equation*}
By the comparison principle from \cite{vBDC}, $w\geq v$ where $v$ is the solution of the previous problem. But, according to \cite{vBPM}, the solution $v$ blows up in finite time. A contradiction between  $w\geq v$ and the global existence of $w$. Thus, $w$ can not exist on $[b,b+1] \times (0,\infty)$, and the solution $u$ exists only in a finite interval. 
\end{proof}

For $1<p\leq 2$, we do not have the blowing-up argument and we are not sure that the existence interval of the solution is finite. 

\begin{theorem}
Let $p \in (1,2]$ and $\lambda \in \mathbb{R}$. For some $\alpha \in (0,\infty]$, there exists a positive and unbounded solution of the Equation \ref{B_param}
\begin{displaymath}
u'' -uu'+u|u|^{p-1}-\lambda u = 0 \textrm{ in } (0,\alpha) .
\end{displaymath}
satisfying 
\begin{displaymath}
u(-a)=0 \textrm{ and } \lim_{ x\to a} u(x) = \infty.
\end{displaymath}
\end{theorem}

With some assumption on the parameter $\lambda$, we can also build a trajectory $(u,v)$ with an initial data $(u_0,0)$ belonging to the abscissa axis.

\begin{lemma}
Let $p \in (1,3)$ and $\lambda \in \mathbb{R}$. Suppose that there exists $\beta>1$ such that
\begin{equation}\label{H_lambda}
\lambda > \max \Bigg\{ \frac{\beta-1}{2\beta}\Bigg( \frac{2\beta^2}{\beta-1} \Bigg)^\frac{1}{3-p} ,\ \beta\Bigg( \frac{2\beta^2}{\beta-1} \Bigg)^\frac{p-1}{3-p} \Bigg\}
\end{equation}
If 
\begin{equation}\label{H_u0}
u_0 = \Bigg( \frac{2\beta^2}{\beta-1} \Bigg)^\frac{1}{3-p},
\end{equation}
then the trajectory $(u,v)$ is not bounded.
\end{lemma}
\begin{proof}
We use the same method as in Lemma \ref{lem_nonborne}: we prove that, under hypotheses \ref{H_lambda} and \ref{H_u0}, the trajectory $(u,v)$  always lies above the curve $\{ v = \beta u^{p-1} -\lambda \}$. Ab absurdo, suppose that there exist $x_* >0$ such that $u(x_*) = u_1$ and $v_1=v(x_*)$ verify
\begin{equation}\label{H_u1}
v_1 = \beta u_1^{p-1} -\lambda ,
\end{equation}
and  
\begin{displaymath}
v(x_*) > \beta u(x_*)^{p-1} -\lambda \ \forall \ 0 < x < x_*.
\end{displaymath}
Thus, in $[0,x_*)$, we have
\begin{equation}\label{H_u2}
\frac{\lambda - u^{p-1}}{v} \geq \frac{-1}{\beta} .
\end{equation}
Equation \ref{variation} gives
\begin{displaymath}
\frac{d v}{d u}= u +u\frac{\lambda -u^{p-1}}{v},
\end{displaymath}
and condition \ref{H_u2} implies
\begin{equation*}
\frac{d v}{d u} \geq \frac{\beta-1}{\beta} u \geq 0.
\end{equation*}
Integration between $u_0$ and $u_1$ leads to
\begin{displaymath}
v(u_1) \geq \frac{\beta-1}{2\beta} (u_1 -u_0) ,
\end{displaymath}
definition \ref{H_u1} gives
\begin{displaymath}
\beta u_1^{p-1} -\lambda \geq \frac{\beta-1}{2\beta} (u_1 -u_0) ,
\end{displaymath}
and we obtain
\begin{equation}\label{H_u4}
u_1^{p-1} \Big( 1 -\frac{\beta-1}{2\beta^2} u_1^{3-p}\Big) \geq \frac{1}{\beta}\Big( \lambda- u_0\frac{\beta-1}{2\beta}\Big) .
\end{equation}
Since of $u_0<u_1$, Equations \ref{H_lambda} and \ref{H_u0} imply 
\begin{displaymath}
\lambda- u_0\frac{\beta-1}{2\beta} >0 \textrm{ and } 1 -\frac{\beta-1}{2\beta^2} u_1^{3-p}<0.
\end{displaymath}
Hence, Equation \ref{H_u4} is a contradiction.
\end{proof}

Concerning Equation \ref{B_param}, and reasoning as in Theorem \ref{thm_nonborne}, we obtain the following result.

\begin{theorem}\label{thm_nonborne2}
Let $p \in (1,3)$ and $\lambda \in \mathbb{R}$ verifying Equation \ref{H_lambda}. For some $\alpha\in (0,\infty]$, there exists a positive and unbounded solution of the Equation
\ref{B_param}
\begin{displaymath}
u'' -uu'+u|u|^{p-1}-\lambda u = 0 \textrm{ in } (0,\alpha) .
\end{displaymath}
satisfying  
\begin{displaymath}
u'(0)=0 \textrm{ and } \lim_{ x\to \alpha} u(x) = \infty.
\end{displaymath}
In addition, if $p\in (2,3)$, then $\alpha$ is finite.
\end{theorem}

\subsection{Limiting case $p=1$}
In this paragraph, we study the case where the exponent $p$ attains the limit $1$. Then, Equation \ref{B_param} becomes
\begin{displaymath}
u'' -uu'+(1-\lambda) u = 0 \textrm{ in } \mathbb{R} ,
\end{displaymath}
and the System \ref{Bstat_sys} is written 
\begin{eqnarray}\label{Bstat_sys_1}
\left( \begin{array}{c}
u'  \\
v'
\end{array} \right)   
=
\left( \begin{array}{c}
v\\
u(v+\lambda -1) \end{array} \right) .
\end{eqnarray}
For $\lambda \not =1$, $(0,0)$ is the only equilibrium point, while for $\lambda  =1$ the axis $\{v=0\}$ is a continuum of equilibria. We begin with the case $\lambda  =1$. Here, we have $\frac{dv}{du}=u$, then 
\begin{displaymath}
v(u)=\frac{1}{2}u^2+c,
\end{displaymath}
where $c$ depends on the initial data. Thus, the phase plane is easily drawn, see Figure \ref{phase3}. Now, suppose $\lambda \not =1$. One can compute the explicit trajectory 
\begin{equation*}
\left\{
   						 \begin{array}{l}
        				u_e(x)=(1-\lambda)x  \\
        				v_e(x)=(1-\lambda) 
    						\end{array}
			  \right. \forall \ x\in \mathbb{R}
\end{equation*}
Then, using the following equations
\begin{displaymath}
\frac{dv}{du} = u+\frac{u}{v}(\lambda-1)  \ \textrm{ and } \ \frac{d^2v}{du^2} = 1+ +\frac{\lambda-1}{v^2}\Big( v-u\frac{dv}{du} \Big)
\end{displaymath}
we can draw the phase plane of the System \ref{Bstat_sys_1}, see Figure \ref{phase3}.  

\begin{figure}[h]
\begin{center}
  \includegraphics[width=0.9\textwidth]{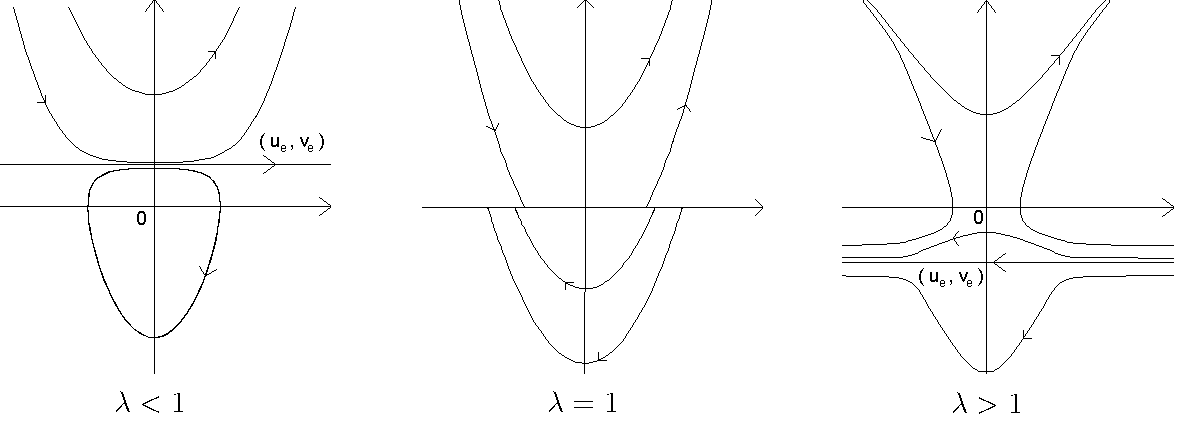}\\
  \caption{Phase planes for  $p=1$.}\label{phase3}
  \end{center}
\end{figure}

\subsection{Bifurcation}
According to the previous paragraphs, we can state that there exists a bifurcation of the phase plane of the System \ref{Bstat_sys}. First, we note that, for a fixed exponent $p$, the value of $\lambda$ influences the phase plane of the System \ref{Bstat_sys}: for $\lambda>0$, the System \ref{Bstat_sys} admits three equilibrium points (a saddle point, an attractive equilibrium and a repulsive equilibrium). The distance between these equilibria goes to $0$ when $\lambda \to 0$, and for $\lambda = 0$, they collapse and generate a unique center, which persists for all negative $\lambda$ (see Figure \ref{Bif1}).

\begin{figure}[h]
\begin{center}
  \includegraphics[width=3in]{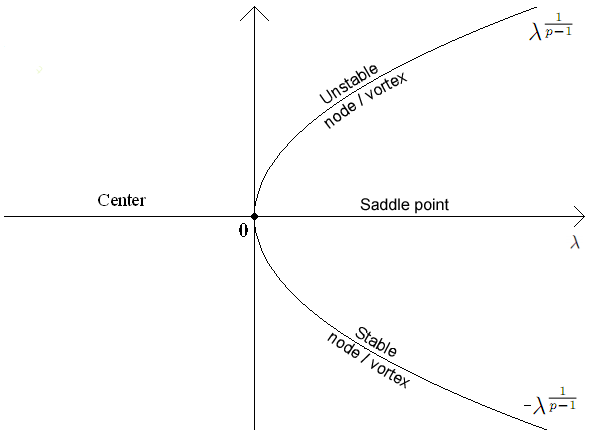}\\
  \caption{Bifurcation of the phase plane of the System \ref{Bstat_sys} with different parameters.}\label{Bif1}
  \end{center}
\end{figure}

Now, for a fixed $\lambda$, the value of the exponent $p$ has an important role. With $\lambda$, the value of $p$ governs the type of the equilibrium points (node, improper node, vortex). The exponent $p$ also establishes if all the trajectories of the System \ref{Bstat_sys} are bounded ($p\geq 3$) or if there exists unbounded trajectories  ($1\leq p<3$). Moreover, when $p$ attains the limit $1$, the critical value of $\lambda$ changes from $0$ (if $p>1$) to $1$ (for $p=1$). The case $\lambda=1$ is special because when $p \to 1$, the three equilibria of the System \ref{Bstat_sys} (a saddle point, an attractive vortex and a repulsive vortex) generate a continuum of equilibria when $p$ attains the limit $1$ (see Figure \ref{Bif2}).

\begin{figure}[!h]
\begin{center}
  \includegraphics[width=0.9\textwidth]{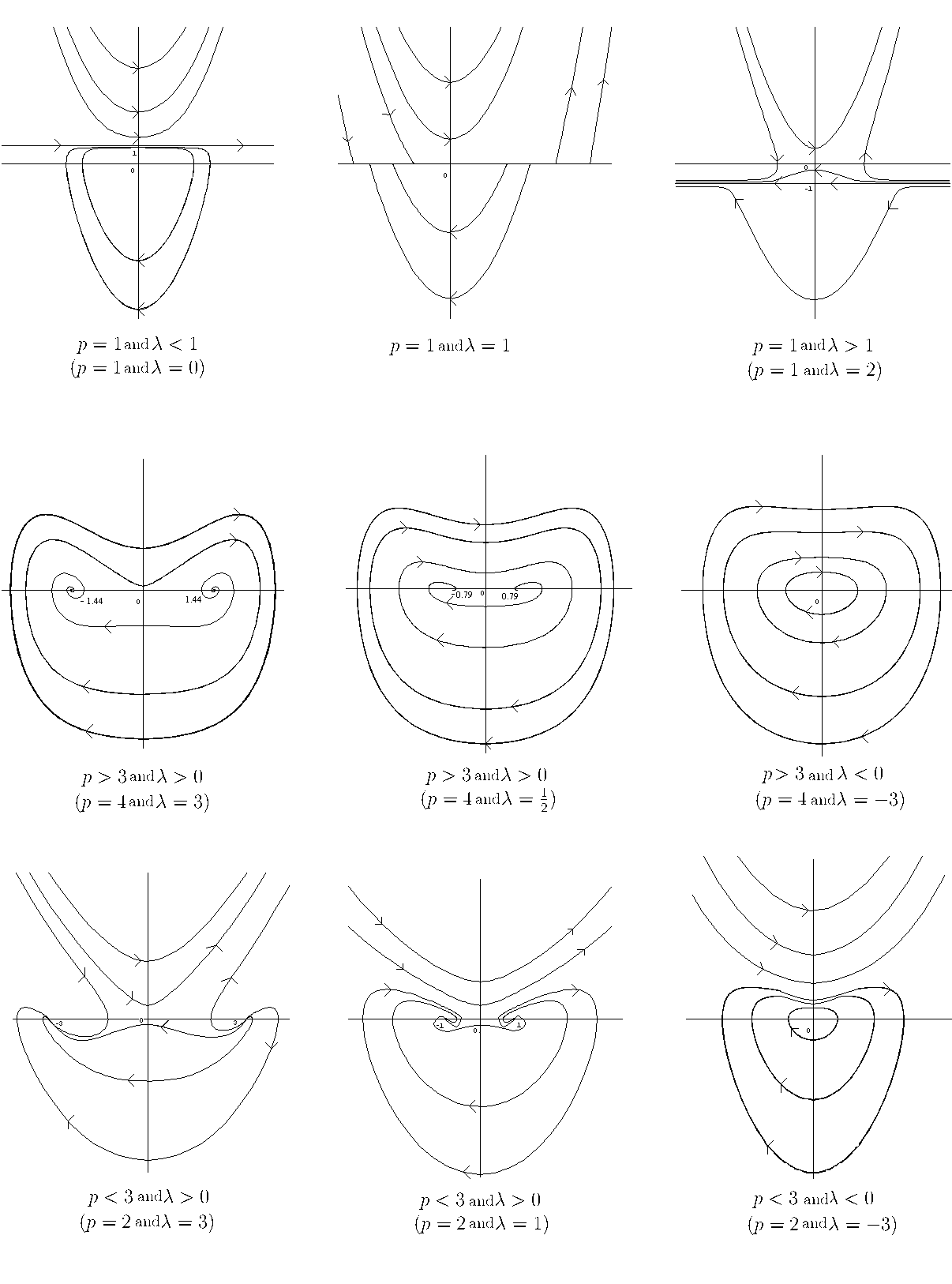}\\
  \caption{Phase planes of the System \ref{Bstat_sys} with different parameters.}\label{Bif2}
  \end{center}
\end{figure}

\section{Parabolic problem}

In this section, we study the parabolic Problem \ref{B_param} for many boundary conditions. First, we use the results concerning the stationary Equation \ref{B_param} when the domain $\Omega$ is bounded. Then, we consider the case of unbounded domains: we investigate global existence using the comparison method, and blow-up phenomenon thanks to a $L^1-$norm technique.

\subsection{Comparison}
We begin with the Dirichlet problem
\begin{equation}\label{Dirichlet_param}
\left\{
    \begin{array}{ll}
        \partial _t u = \partial_x^2 u -u\partial_x u + u^p - \lambda u& \textrm{ in }[-\alpha,\alpha] \times (0,\infty), \\
       u= 0 & \textrm{ on } \{\pm \alpha\} \times (0,\infty), \\
        u(\cdot , 0) = \varphi & \textrm{ in } [-\alpha,\alpha] ,
    \end{array}
\right.
\end{equation}
where $\alpha>0$, $p>1$, $\lambda \in \mathbb{R}$ and $\varphi \in \mathcal{C}_0([-\alpha,\alpha])$ is non-negative. Thanks to the comparison principle \cite{vBDC} and with the results of the previous sections, we have:

\begin{theorem}
Let  $p>1$ and $\lambda \in \mathbb{R}$. For some $\alpha>0$, there exists a global positive solution 
\begin{displaymath}
u \in \mathcal{C}([-\alpha,\alpha] \times [0,\infty)) \cap \mathcal{C}^{2,1}([-\alpha,\alpha] \times (0,\infty))
\end{displaymath}
of Problem \ref{Dirichlet_param} if the initial data $\varphi \in \mathcal{C}_0([-\alpha,\alpha])$ is sufficiently small.
\end{theorem}
\begin{proof}
If $p\geq3$ and $\lambda>0$, consider $\beta \in \mathcal{C}^2([-\alpha,\alpha])$ a solution of \ref{B_param} with the Dirichlet boundary conditions (see Theorem \ref{sol+}). Suppose that $\varphi$ is small enough: $\varphi \leq \beta$ in $[-\alpha,\alpha]$. Then, we obtain
\begin{equation*}
\left\{
    \begin{array}{ll}
        \partial _t \beta = 0 =\partial_x^2 \beta -\beta\partial_x \beta + \beta^p - \lambda u& \textrm{ in }[-\alpha,\alpha] \times (0,\infty), \\
       \beta= 0 & \textrm{ on } \{\pm \alpha\} \times (0,\infty), \\
        \beta(\cdot , 0) \geq \varphi & \textrm{ in } [-\alpha,\alpha] .
    \end{array}
\right.
\end{equation*}
Using the comparison principle from \cite{vBDC}, we prove that there exists a solution $u$ of \ref{Dirichlet_param} satisfying $0\leq u \leq \beta$ for all $(x,t) \in [-\alpha,\alpha] \times (0,\infty)$. Thus, $u$ is a global positive solution. If  $1<p<3$ and $\lambda>0$, then we just need to choose a positive solution $\beta$ given in Theorem \ref{p<3_sol+} (even if $\beta(\pm \alpha)>0$). For $\lambda \leq 0$, we consider the Dirichlet solution given in Theorem \ref{lneg_sol+}.
\end{proof}

Now, we replace the Dirichlet boundary conditions by the dynamical boundary conditions. Consider the following problem 
\begin{equation}\label{Dynamical_param}
\left\{
    \begin{array}{ll}
        \partial _t u = \partial_x^2 u -u\partial_x u + u^p - \lambda u& \textrm{ in }[-\alpha,\alpha] \times (0,\infty), \\
        \sigma \partial _t u + \partial _\nu u= 0 & \textrm{ on } \{\pm \alpha\} \times (0,\infty), \\
        u(\cdot , 0) = \varphi & \textrm{ in } [-\alpha,\alpha] ,
    \end{array}
\right.
\end{equation}
with  $\alpha>0$, $p>1$, $\lambda \in \mathbb{R}$ and where $\varphi \in \mathcal{C}([-\alpha,\alpha])$ and $ \sigma \in \mathcal{C}^1(\{\pm \alpha\} [0,\infty))$ are non-negative. We obtain two results, depending on the sign of $\lambda$.

\begin{theorem}\label{full_comparison}
Let  $p>1$ and $\lambda >0$. There exists a global positive solution 
\begin{displaymath}
u \in \mathcal{C}([-\alpha,\alpha] \times [0,\infty)) \cap \mathcal{C}^{2,1}([-\alpha,\alpha] \times (0,\infty))
\end{displaymath}
of Problem \ref{Dynamical_param} 
\begin{itemize}
\item for all $\alpha>0$ if $\varphi \leq \lambda^\frac{1}{p-1}$.
\item for some $\alpha>0$ if $\varphi-\lambda^\frac{1}{p-1}$ is sign changing and $\max \{\varphi-\lambda^\frac{1}{p-1},0\}$ is sufficiently close to $0$.
\item for no $\alpha>0$ if $\varphi>\lambda^\frac{1}{p-1}$ and $p>2$.
\end{itemize}
\end{theorem}
\begin{proof}
For the first statement, we just need to note that the constant function $\lambda^\frac{1}{p-1}$ is a super-solution of \ref{Dynamical_param} when $0 \leq \varphi \leq \lambda^\frac{1}{p-1}$. 
For the second statement, we consider two cases: when $p\geq3$, we consider a positive solution $w$ of \ref{B_param} under the Neumann boundary conditions, see Theorem \ref{sol+}. Choosing $0 \leq \varphi \leq w$, we obtain a super-solution of \ref{Dynamical_param}. If $1<p<3$, we consider  a trajectory $(\mu,\gamma)$ of \ref{Bstat_sys} with $0<\mu(0)<\lambda^\frac{1}{p-1}$ and $\gamma(0)=0$. According to Equation \ref{variation}, for a small $x_*>0$, we have $\gamma(-x) <0$ and $\gamma(x) >0$ for all $x\in (0,x_*)$. Thus, $\mu$ satisfies $\partial_\nu \mu(-x_*) = -\gamma(-x_*)>0$ and $\partial_\nu \mu(x_*) = \gamma(x_*)>0$, and it is a super-solution of \ref{Dynamical_param} when $0 \leq \varphi \leq \mu$ in $[-x_*, x_*]$.
Then, using these super-solutions and the comparison principle from \cite{vBDC}, we prove first and second assertions. For the third statement, consider
$c>0$ such that
\begin{displaymath}
\varphi > c > \lambda^\frac{1}{p-1} .
\end{displaymath}
The comparison principle from \cite{vBDC} implies that $u>c$, where $u$ denote the solution of \ref{Dynamical_param} with the initial data $\varphi$. Hence, there exists $d>0$  such that
\begin{displaymath}
u^p - \lambda u \geq d u^p \textrm{ for all }x\in [-\alpha,\alpha] \textrm{ and for all } t>0 .
\end{displaymath}
Thus, $u$ verifies
\begin{eqnarray*}
\left\{
\begin{array}{ll}
\partial_t u \geq \partial_x^2 u-u\partial_x u+d u^p & \textrm{ in } [-\alpha,\alpha] \textrm{ for } t>0,\\
\sigma \partial_t u +\partial_\nu u =0 & \textrm{ on } \{-\alpha,\ \alpha\} \textrm{ for } t>0,\\
u(\cdot ,0)> c>0 & \textrm{ in } [-\alpha,\alpha] .  \\
\end{array}
\right.
\end{eqnarray*}
Then,  blow-up results from \cite{vBPM} imply the blowing-up in finite time of $u$.
\end{proof}

\begin{theorem}
Let  $p>2$ and $\lambda \leq 0$. For all $\alpha>0$, the positive solution $u$ of Problem \ref{Dynamical_param} blows up in finite time if the initial data $\varphi \in \mathcal{C}([-\alpha,\alpha])$ satisfies
\begin{equation*}
\forall \ x \in [-\alpha,\alpha] , \ \varphi(x) >0.
\end{equation*}
\end{theorem}
\begin{proof}
Since of $\lambda \leq 0$, the function $u$ verifies
\begin{eqnarray*}
\left\{
\begin{array}{ll}
\partial_t u \geq \partial_x^2 u-u\partial_x u+ u^p & \textrm{ in } [-\alpha,\alpha] \textrm{ for } t>0,\\
\sigma \partial_t u +\partial_\nu u =0 & \textrm{ on } \{-\alpha,\ \alpha\} \textrm{ for } t>0,\\
u(\cdot ,0)>0 & \textrm{ in } [-\alpha,\alpha] .  \\
\end{array}
\right.
\end{eqnarray*}
Thanks to the blow-up results from \cite{vBPM}, we know that $u$ blows up in finite time.
\end{proof}

\begin{remark}
The Neumann boundary conditions are included here, with the special case $\sigma \equiv 0$.
\end{remark}

\subsection{Global existence in unbounded domains}
We study the Problem \ref{B_param} under the Dirichlet, the Neumann and the dynamical boundary conditions when $\Omega$ is an unbounded domain. Using some explicit super-solutions, we look for global existence in the three types of unbounded domains: $(-\infty,0)$, $(0,\infty)$ and $\mathbb{R}$. We begin with the case $\lambda>0$:

\begin{theorem}
Let $p>1$, $\lambda>0$ , $\varphi \in \mathcal{C}(\overline{\Omega})$ a non-negative function, and let $\Omega$ be any unbounded domain. Then, the Problem \ref{B_param} admits a global positive solution if the initial data satisfies
\begin{displaymath}
0\leq \varphi \leq \lambda^\frac{1}{p-1},
\end{displaymath}
and when $\mathcal{B}(u)=0$ stands for the Dirichlet, the Neumann, the Robin ($\partial_\nu u + a u =0$ with $a \geq 0$) or the dynamical boundary conditions.
\end{theorem}
\begin{proof}
As in the proof of Theorem \ref{full_comparison}, we consider the constant function $v(x,t)=\lambda^\frac{1}{p-1}$ for all $(x,t)\in \Omega \times (0,\infty)$. Then, $v$ satisfies Burger's Equation, the choice of $\varphi$ implies  $\varphi \leq v(\cdot,0)$ in $\Omega$. On the boundary, we have:
\begin{eqnarray*}
\left.
\begin{array}{lll}
v &\geq 0 & \textrm{ (Dirichlet). } \\
\partial_\nu v &= 0 & \textrm{ (Neumann). } \\
\partial_\nu v + a v & \geq 0 & \textrm{ (Robin). } \\
\sigma \partial_t v + \partial_\nu v &= 0 & \textrm{ (Dynamical). }
\end{array}
\right.
\end{eqnarray*}
Thus, $v$ is super-solution of  \ref{B_param} for the four boundary conditions above, and we conclude with the comparison principle \cite{vBDC}.
\end{proof}

If $\lambda \leq0$, we must add some restrictions, and we obtain the following results.

\begin{theorem}
Assume $\Omega=(0,\infty)$ and let  $p\in (1,2]$, $\lambda \leq 0$ and $\varphi \in \mathcal{C}(\overline{\Omega})$ a non-negative function. Then, the Problem \ref{B_param} admits a global positive solution if the initial data is bounded and when $\mathcal{B}(u)=0$ stands for the Dirichlet boundary conditions or the dynamical boundary conditions with $\sigma>0$ constant.
\end{theorem}
\begin{proof}
We deal with the comparison principle \cite{vBDC} and the explicit function $v(x,t)=Ae^{\alpha x+(t+t_0)^2}$ defined in $\mathbb{R}^+ \times \mathbb{R}^+$. Computing the partial derivatives, we have
\begin{eqnarray*}
\left.
\begin{array}{lll}
\partial_t   v(x,t) & = & 2(t+t_0)v  .\\
\partial_x   v(x,t) & = & \alpha v  .\\
\partial_x^2 v(x,t) & = & \alpha^2 v  .
\end{array}
\right.
\end{eqnarray*}
Choosing $t_0 \geq \frac{1}{2}\Big( \alpha^2- \lambda\Big)$, we obtain
\begin{displaymath}
\partial_t v - \partial_x^2 v +v \partial_x v - v^p +\lambda v \geq  v^2 \Big( \alpha- v^{p-2}  \Big).
\end{displaymath}
Thanks to $p\leq 2$ and with $\alpha x+(t+t_0)^2 \geq 0$ in $\mathbb{R}^+ \times \mathbb{R}^+$, we have $v^{p-2} \leq A^{p-2}$. Choosing $A^{p-2} \leq \alpha$, we obtain $\partial_t v - \partial_x^2 v +v\partial_x v - v^p +\lambda v \geq 0$. Since $v\geq 0$, the case of the Dirichlet boundary conditions is trivial. Choosing $t_0 \geq \frac{\alpha}{2\sigma}$, the case of the dynamical boundary conditions is verified thanks to
\begin{displaymath}
\sigma \partial_t v + \partial_nu v = v \Big( 2\sigma(t+t_0) -\alpha \Big) \geq 0.
\end{displaymath}
Finally, we have a super-solution choosing $A \geq \sup \varphi$.
\end{proof}

\begin{remark}
In the previous proof, one can see that the dynamical boundary conditions are satisfied for a more general coefficient $\sigma$ verifying
\begin{displaymath}
\sigma(x,t) \geq \frac{\alpha}{2(t+t_0)}.
\end{displaymath}
And replacing the function $v$ by $w(x,t)=Ae^{\alpha x+(t+t_0)^n}$, we can consider smaller coefficients $\sigma>0$ with $\sigma (x,t)  \underset{t \to \infty}{\sim} t^{-n+1}$.
\end{remark}

\begin{corollary}
Suppose $\Omega=(-\infty,0)$ or $\Omega=\mathbb{R}$. Let $p =2$, $\lambda\leq 0$ and $\varphi \in \mathcal{C}(\Omega)$. Then the Problem \ref{B_param} admits a global positive solution if there exists $C>0$ and $\alpha>0$ such that
\begin{displaymath}
0\leq \varphi(x) \leq Ce^{ax} \textrm{ in } \Omega
\end{displaymath}
and when $\mathcal{B}(u)=0$ stands for the Dirichlet, the Neumann or the dynamical boundary conditions with $\sigma>0$. 
\end{corollary}
\begin{proof}
As in the previous theorem, we consider $v(x,t)=Ae^{\alpha x+(t+t_0)^2}$. Thanks to $p=2$ and with some appropriate constants $A$ and $\alpha$, we have
\begin{eqnarray*}
\left\{
\begin{array}{lll}
\partial_t v - \partial_x^2 v +v\partial_x v - v^p +\lambda v  \geq  0 & \textrm { in } & \Omega \times [0,\infty). \\
v(\cdot,0) \geq  \varphi & \textrm { in } & \Omega.
\end{array}
\right.
\end{eqnarray*}
The case $\Omega=\mathbb{R}$ (no boundary) and the case of Dirichlet boundary conditions are trivial. For $\Omega=(-\infty,0)$, we have $\partial_\nu v =\partial_x v = \alpha v > 0$ on the boundary. Thus, the Neumann boundary conditions and the dynamical boundary conditions with $\sigma \geq 0$ are verified.
\end{proof}

When $\lambda = 0$, $\Omega = (-\infty,0)$ and  $p>3$, the Green's function of the heat equation is a suitable super-solution for the Problem \ref{B_param}.

\begin{theorem}
Assume $\Omega=(-\infty,0)$, $p >3$ and $\varphi \in \mathcal{C}(\Omega)$. Then the Problem \ref{B_param} admits a global positive solution
if the initial data $\varphi$ is sufficiently small and when $\mathcal{B}(u)=0$ stands for the Dirichlet, the Neumann or the dynamical boundary conditions with $\sigma >0$ constant. 
\end{theorem}
\begin{proof}
Consider the function $v(x,t)= A(t+1)^{-\gamma} e^\frac{-(x+y)^2}{4t+4}$ defined in $\mathbb{R}^- \times \mathbb{R}^+$ with $A>0$, $\gamma =\frac{1}{p-1}$ and $y=-2\sigma\gamma$. A simple calculus leads to
\begin{displaymath}
\partial_t v - \partial_x^2 v +v\partial_x v - v^p +\lambda v  = \frac{v}{2(t+1)} \Big( -2\gamma +1 -(x+y)v - v^{p-1} \Big).                           
\end{displaymath}
By definition of $\gamma$ and $p>3$, we have $-2\gamma +1>0$. Since $v^{p-1} \leq A^{p-1}$, and because $-(x+y)>0$ for all $x \in \Omega$, we obtain $\partial_t v - \partial_x^2 v +v\partial_x v - v^p +\lambda v \geq 0$ by choosing $A$ small enough. The case of the Dirichlet boundary conditions is clear because $v\geq 0$. For the dynamical boundary conditions and the Neumann boundary conditions ($\sigma \equiv0$), we use the definition of $y$ and we have
\begin{displaymath}
\sigma \partial_t v + \partial_\nu v  \geq \frac{v}{2(t+1)}\Big( -2\sigma \gamma -y\Big) \geq 0 .
\end{displaymath}
Thus, $v$ is a super-solution of the Problem \ref{B_param} as soon as $\varphi \leq v(\cdot,0)$ in $\Omega$.\\
\end{proof}

\subsection{Blow up in unbounded domains}
Here, using some weighted $L^1-$norms, we examine blow-up phenomena for some solutions of Problem \ref{B_param} in unbounded domains satisfying the Neumann, the Robin, and some nonlinear boundary conditions and this growth order at infinity: for all $a>0$ and for all $t>0$
\begin{equation}\label{growth_infty}
\lim_{|x| \to \infty} u(x,t)e^{-a |x|} = 0 \textrm{ and } \lim_{|x| \to \infty} |\partial_x u(x,t)|e^{-a |x|} = 0 .
\end{equation}
Unless otherwise stated, we always suppose $\Omega= (0,\infty)$. We begin with a lemma which gives a criterion for the blowing-up of the solution. 

\begin{lemma}\label{blowup_lem}
Let $u$ be a solution of Problem \ref{B_param} satisfying the condition \ref{growth_infty}. If there exists $\alpha>0$ such that 
\begin{displaymath}
N_\alpha(t) := \int_0^\infty u(x,t) e^{-\alpha x} \ dx
\end{displaymath}
blows-up in finite time, then $u$ also blows-up in finite time.
\end{lemma}
\begin{proof}
Consider $\alpha>0$ such that $N_\alpha$ blows-up in finite time. Using the following inequality
\begin{displaymath}
N_\alpha(t) \leq \int_0^\infty e^{-\alpha x/2} \ dx \cdot \sup_\Omega u(x,t)e^{-\frac{\alpha}{2}x}  = \frac{2}{\alpha} \sup_\Omega u(x,t)e^{-\frac{\alpha}{2}x} ,
\end{displaymath}
and because $N_\alpha$ blows up, we can deduce the blowing up in finite time of the function $u(x,t)e^{-\frac{\alpha}{2}x}$. Then, thanks to the growth order condition \ref{growth_infty}, the solution $u$ must blow up too. 
\end{proof}

We also need this technical lemma.

\begin{lemma}\label{lem_tech}
Let $u$ be a solution of Problem \ref{B_param}. Then, for all $\tau>0$ there exists $c>0$ such that 
\begin{displaymath}
u(0,t)\geq c \ \textrm{ for all  } \ t\geq\tau \ .
\end{displaymath}
\end{lemma}
\begin{proof}
Let $v$ be the positive solution of the following problem
\begin{eqnarray*}
\left\{
\begin{array}{lll}
\partial_t v = \partial_x^2 v -v\partial_x v + v^p -\lambda v   & \textrm { in } & [0,1] \times [0,\infty), \\
\mathcal{B}(v)  =  0 & \textrm { on } & \{0\} \times [0,\infty), \\
v =  0 & \textrm { on } & \{1\} \times [0,\infty), \\
v(\cdot,0) =  \varphi_1 & \textrm { in } & [0,1],
\end{array}
\right.
\end{eqnarray*}
where $\mathcal{B}(v)  =  0$ is the same boundary condition as in Problem \ref{B_param}, where $\varphi_1 \in \mathcal{C}^2([0,1])$ satisfies  $\varphi_1(1)=0$, $\partial_x^2 \varphi_1 -\varphi_1 \partial_x \varphi_1 + \varphi_1^p -\lambda \varphi_1 \geq 0$ and $0\leq \varphi_1 \leq \varphi$ in $[0,1]$. Thanks to $u(\cdot,0) \geq v(\cdot,0)$ in $[0,1]$ and $u(1,t) \geq 0=v(1,t)$ for all $t>0$, the comparison principle from \cite{vBDC} implies 
\begin{displaymath}
u(x,t)\geq v(x,t) \textrm{ for all } x \in[0,1] \textrm{  and } t>0 .
\end{displaymath}
Then, the comparison principle and the maximum principle from \cite{vBDC} imply 
\begin{displaymath}
\partial_t v(x,t) \geq 0 \textrm{ and  }  v(x,t) > 0 .
\end{displaymath}
for all $x \in[0,1]$ and $t>0$, see Lemma 2.1 in \cite{vBPM2}. Thus, for all $\tau>0$, we obtain 
\begin{displaymath}
u(0,t) \geq v(0,t) \geq v(0,\tau)>0 \textrm{ for all  }  t \geq \tau .
\end{displaymath}
\end{proof}

\begin{theorem}\label{main_N-thm}
Let $\lambda <0$ and $p\geq2$. Then the Problem \ref{B_param} admits no global positive solution when $\mathcal{B}(u)=0$ stands for the Neumann boundary conditions.
\end{theorem}
\begin{proof}
We aim to prove the existence of $\alpha>0$ and $\beta>0$ such that $N'_\alpha \geq \beta N_\alpha^p$ where 
\begin{displaymath}
N_\alpha(t) := \int_0^\infty u(x,t) e^{-\alpha x} \ dx
\end{displaymath}
Derivating the function $N_\alpha$, we obtain
\begin{equation*}
 \begin{split}
    N'_\alpha(t) =& \int_0^\infty \partial_t u(x,t) e^{-\alpha x} \ dx \\
                 =& \int_0^\infty \Big(\partial_x^2 u(x,t)\Big)e^{-\alpha x} \ dx -\int_0^\infty \Big(u(x,t)\partial_x u(x,t)\Big)e^{-\alpha x} \ dx\\
                  & +   \int_0^\infty  u^p(x,t) e^{-\alpha x} \ dx -\lambda \int_0^\infty  u(x,t) e^{-\alpha x} \ dx			.
 \end{split}
\end{equation*}
Using the growth order condition \ref{growth_infty} and integrating by parts, we obtain
\begin{displaymath}
\int_0^\infty \Big(\partial_x^2 u(x,t)\Big)e^{-\alpha x} \ dx = \alpha^2 \int_0^\infty  u(x,t)e^{-\alpha x} \ dx + \partial_\nu u(0,t) -\alpha u(0,t)
\end{displaymath}
and 
\begin{displaymath}
\int_0^\infty \Big(u(x,t)\partial_x u(x,t)\Big)e^{-\alpha x} \ dx = \frac{\alpha}{2} \int_0^\infty  u^2(x,t)e^{-\alpha x} \ dx  -\frac{u^2(0,t)}{2} .
\end{displaymath}
Thus, we have
\begin{equation}\label{derive_N}
 \begin{split}
N'_\alpha(t) = & \int_0^\infty  u(x,t) e^{-\alpha x} \Big( \alpha^2 - \frac{\alpha}{2}u(x,t) - \lambda + u^{p-1}(x,t) \Big) \ dx \\
						   &  + \partial_\nu u(0,t) -\alpha u(0,t) +\frac{u^2(0,t)}{2}.
 \end{split}
\end{equation}
Thanks to Lemma \ref{lem_tech}, and considering $u$ from a time $\tau>0$, we can assume that 
\begin{displaymath}
c:=\min_{t>0} u(0,t)>0 \ .
\end{displaymath}
Then, if $\alpha$ is small enough ($\alpha \leq c/2$), we have $-\alpha u(0,t) +\frac{u^2(0,t)}{2} \geq 0$. Then, the Neumann boundary conditions imply
\begin{equation}\label{minor_N}
N'_\alpha(t) \geq   \int_0^\infty  u(x,t) e^{-\alpha x} \Big( \alpha^2 - \frac{\alpha}{2}u(x,t) - \lambda + u^{p-1}(x,t) \Big) \ dx .
\end{equation}
Shrinking $\alpha$, we can suppose $\alpha \leq -2\lambda$ and $\alpha \leq 1$. When $u(x,t) \leq 1$,  we have $-\lambda - \alpha u(x,t)/2 >0$. On the other hand,  if $u(x,t) \geq 1$, we have $u^{p-1}(x,t) -\alpha u(x,t)/2 \geq u^{p-1}(x,t)/2$. Hence, we obtain: 
\begin{displaymath}
N'_\alpha(t) \geq  \frac{1}{2} \int_0^\infty  u^p(x,t) e^{-\alpha x}  \ dx .
\end{displaymath}
H\"older inequality 
\begin{displaymath}
\int_0^\infty  u(x,t) e^{-\alpha x}  \ dx \leq  \Bigg( \int_0^\infty  u^p(x,t) e^{-\alpha x}  \ dx \Bigg) ^\frac{1}{p} \Bigg( \int_0^\infty   e^{-\alpha x}  \ dx \Bigg) ^\frac{p-1}{p}
\end{displaymath}
leads  to $N'_\alpha(t) \geq \beta N^p_\alpha(t)$ with 
\begin{displaymath}
\beta = \frac{1}{2} \Bigg( \int_0^\infty   e^{-\alpha x}  \ dx \Bigg) ^{1-p} .
\end{displaymath}
Finally, we prove the blowing-up of $N_\alpha$ in finite time. Integrating the differential inequality $N'_\alpha(t) \geq \beta N^p_\alpha(t)$ between $0$ and $t>0$, we obtain
\begin{displaymath}
\frac{1}{1-p}\Big( N_\alpha^{1-p}(t)- N_\alpha^{1-p}(0) \Big) =  \int_{N_\alpha(0)} ^{N_\alpha(t)} s^{-p} \ ds  = \int_0^t \frac{N_\alpha'(t)}{N_\alpha^p(t)} \ dt \geq \beta t ,
\end{displaymath}
and 
\begin{displaymath}
N_\alpha(t) \geq \Big(N_\alpha^{1-p}(0) -(p-1)\beta t \Big)^\frac{-1}{p-1} .
\end{displaymath}
Since of $\frac{-1}{p-1}<0$, the right hand side term blows up at $t=\frac{N_\alpha^{1-p}(0)}{(p-1)\beta}>0$. We conclude with Lemma \ref{blowup_lem}.
\end{proof}

\begin{corollary}
Let $\lambda <0$ and $p\geq2$. Then the Problem \ref{B_param} admits no global positive solution when $\mathcal{B}(u)=0$ stands for the nonlinear boundary conditions $\partial_\nu u = g(u)$, where $g$ is a function such that there exists $\delta>0$ and $\varepsilon\leq 1/2$ satisfying
\begin{displaymath}
g(\eta)  \geq \delta \eta - \varepsilon \eta^2 .
\end{displaymath}  
\end{corollary}
\begin{proof}
We follow the proof of Theorem \ref{main_N-thm}. We just change the choice of $\alpha$: let $\alpha>0$ such that $\alpha \leq \delta$, and use the following minoration in Equation \ref{derive_N}:
\begin{equation*}
 \begin{split}
\partial_\nu  u(0,t) - \alpha  u(0,t) +\frac{1}{2} u^2(0,t) = &g(u)  - \alpha  u(0,t) +\frac{1}{2} u^2(0,t)\\
						   \geq & (\delta - \alpha ) u(0,t) +(\frac{1}{2}-\varepsilon) u^2(0,t) \geq 0.
 \end{split}
\end{equation*}
Then, we return to Equation \ref {minor_N}.
\end{proof}

When $\lambda=0$, the choice of $\alpha$ is too strict. Meanwhile, we obtain some blow-up results imposing some restrictions on the exponent $p$ and on the initial data.

\begin{theorem}
Let $\lambda =0$ and $1< p \leq 3$. Then the Problem \ref{B_param} admits no global positive solution when $\mathcal{B}(u)=0$ stands for the Neumann boundary conditions.
\end{theorem}
\begin{proof}
Return to the proof of Theorem \ref{main_N-thm}. Under the Neumann boundary conditions and with $\lambda=0$, Equation \ref{derive_N} becomes
\begin{displaymath}
N'_\alpha(t) =  \int_0^\infty  u(x,t) e^{-\alpha x} \Big( \alpha^2 - \frac{\alpha}{2}u(x,t) + u^{p-1}(x,t) \Big) \ dx  -\alpha u(0,t) +\frac{u^2(0,t)}{2}.
\end{displaymath}
Let $\beta \in (0,1)$ and put it into the previous equation:
\begin{equation*}
 \begin{split}
N'_\alpha(t) = &  \int_0^\infty  u(x,t) e^{-\alpha x} \Big( \alpha^2 - \frac{\alpha}{2}u(x,t) + \beta u^{p-1}(x,t) \Big) \ dx \\
							 & -\alpha u(0,t) +\frac{u^2(0,t)}{2} + (1-\beta)\int_0^\infty  u^p(x,t) e^{-\alpha x}  \ dx.
 \end{split}
\end{equation*}
If $u \leq 2\alpha$, we have $\alpha^2 -\alpha u/2 \geq 0$, whereas $u> 2\alpha$ implies 
\begin{displaymath}
-\frac{\alpha}{2}u + \beta u^{p-1}\geq u \Big(-\frac{\alpha}{2}+\beta(2\alpha)^{p-2}) \Big) .
\end{displaymath}
It is non negative if
\begin{equation}\label{alpha-beta}
\beta \alpha^{p-3} \geq 2^{1-p}.
\end{equation}
Thanks to $1< p \leq 3$, Equation \ref{alpha-beta} is achievied by choosing $\alpha>0$ sufficiently small and $\beta \in (0,1)$ depending on $p$. Thus, we obtain
\begin{displaymath}
N'_\alpha(t) \geq  -\alpha u(0,t) +\frac{u^2(0,t)}{2} + (1-\beta)\int_0^\infty  u^p(x,t) e^{-\alpha x}  \ dx.
\end{displaymath}
Then, we can suppose that $u(0,t)>c>0$ for all $t>0$ (see Lemma \ref{lem_tech}), and with $\alpha <c/2$ we have $-\alpha u(0,t) +\frac{u^2(0,t)}{2}>0$. Hence 
\begin{displaymath}
N'_\alpha(t) \geq  (1-\beta)\int_0^\infty  u^p(x,t) e^{-\alpha x}  \ dx.
\end{displaymath}
As in the proof of Theorem \ref{main_N-thm}, we use H\"older inequality and we are led to $N'_\alpha \geq \delta N^p_\alpha$ with $\delta>0$ depending on $\alpha$, $\beta$ and $p$. Hence, $N_\alpha$ blows-up in finite time, so does the solution $u$, see Lemma \ref{blowup_lem}.
\end{proof}

\begin{theorem}
Let $\lambda =0$ and $p>3$. Then the Problem \ref{B_param} admits no global positive solution when $\mathcal{B}(u)=0$ stands for the Neumann boundary conditions and if the initial data satisfies $\varphi(0)>2^\frac{1-p}{p-3}$.
\end{theorem}
\begin{proof}
The proof is similar to the previous one. Go back to Equation \ref{alpha-beta}: since $p>3$, we must choose $\alpha$ such that
\begin{displaymath}
\alpha \geq 2^\frac{1-p}{p-3}\beta^\frac{-1}{p-3}.
\end{displaymath}
Under this condition, $N_\alpha$ satisfies the differential inequality
\begin{displaymath}
N'_\alpha(t) \geq  -\alpha u(0,t) +\frac{u^2(0,t)}{2} + (1-\beta)\int_0^\infty  u^p(x,t) e^{-\alpha x}  \ dx.
\end{displaymath}
Because $\alpha$ can not be too small, we must use the assumption $\varphi(0)>2^\frac{1-p}{p-3}$. Using Lemma \ref{lem_tech}, we have
\begin{displaymath}
u(0,t) \geq \varphi(0)>2^\frac{1-p}{p-3} \textrm{ , for all } t>0.
\end{displaymath}
Thus, with $\beta$ very close to $1$ and with $\alpha =2^\frac{1-p}{p-3}\beta^\frac{-1}{p-3}$, we obtain $-\alpha u(0,t) +\frac{u^2(0,t)}{2}\geq 0$. Hence, we have
\begin{displaymath}
N'_\alpha(t) \geq  (1-\beta)\int_0^\infty  u^p(x,t) e^{-\alpha x}  \ dx .
\end{displaymath}
We conclude with H\"older inequality and the blowing up of $N_\alpha$.
\end{proof}

\begin{corollary}
Let $\lambda =0$ and $p>3$. Then the Problem \ref{B_param} admits no global positive solution when $\mathcal{B}(u)=0$ stands for the Neumann boundary conditions and if the initial data satisfies 
\begin{equation}\label{Hyp-L0p3}
\int_0^\infty \varphi(x)e^{-x} \ dx > \frac{3p-7}{p-3} \cdot 2^\frac{5-3p}{p-3} \cdot \Bigg( \frac{2p-4}{3p-7} \Bigg)^\frac{4-2p}{p-3}.
\end{equation}
\end{corollary}
\begin{proof}
Return to the proof of Theorem \ref{main_N-thm}. Under the Neumann boundary conditions and introducing $\beta$ and $\delta \in (0,1)$ in Equation \ref{derive_N}, we obtain
\begin{equation*}
 \begin{split}
N'_\alpha(t) = &  \int_0^\infty  u(x,t) e^{-\alpha x} \Big( \delta \alpha^2 - \frac{\alpha}{2}u(x,t) + \beta u^{p-1}(x,t) \Big) \ dx \\
							 & -\alpha u(0,t) +\frac{u^2(0,t)}{2} +(1-\delta)N_\alpha(t) +(1-\beta)\int_0^\infty  u^p(x,t) e^{-\alpha x}  \ dx.
 \end{split}
\end{equation*}
Studying both cases $u\geq 2\alpha \delta$ and $u\leq 2\alpha \delta$, we obtain $\delta \alpha^2 -\alpha u/2 +\beta u^{p-1} \geq 0$ if
\begin{displaymath}
\alpha = 2^\frac{1-p}{p-3} \beta^\frac{-1}{p-3} \delta^\frac{2-p}{p-3}.
\end{displaymath}
Since of $u^2/2-\alpha u \geq -\alpha^2/2$ and using H\"older inequality we have
\begin{equation}\label{N-gamma}
N'_\alpha(t) \geq (1-\delta)N_\alpha(t) + \gamma  N^p_\alpha(t)-\frac{\alpha^2}{2}  , 
\end{equation}
where $\gamma = (1-\beta) \Big( \int_0^\infty   e^{-\alpha x}  \ dx \Big)^{1-p}>0$. First, consider this minoration
\begin{displaymath}
N'_\alpha(t) \geq (1-\delta)N_\alpha(t)  -\frac{\alpha^2}{2}  . 
\end{displaymath}
Thus, $N_\alpha$ satisfies 
\begin{displaymath}
N_\alpha(t) \geq \frac{\alpha^2}{2(1-\delta)}  + A e^{(1-\delta)t}  , \ A\in \mathbb{R}. 
\end{displaymath}
In particular, $N_\alpha(0) \geq \alpha^2(2-2\delta)^{-1}  + A $. With the optimal choice of $\delta = (2p-4)/(3p-7)$ and with $\beta \in (0,1)$ close to $1$, Hypothesis \ref{Hyp-L0p3} implies $N_\alpha(0) > \alpha^2(2-2\delta)^{-1}$. Thus, $A$ is positive and we obtain
\begin{displaymath}
(1-\delta) N_\alpha(t) - \frac{\alpha^2}{2}  \geq 0. 
\end{displaymath}
From Equation \ref{N-gamma}, we deduce 
\begin{displaymath}
 N'_\alpha(t)   \geq \gamma  N^p_\alpha(t). 
\end{displaymath}
Hence $N_\alpha$ blows-up, and the solution $u$ blows up too, see Lemma \ref{blowup_lem}.
\end{proof}

Finally, if $\Omega= (-\infty,0)$, we must change the weight in $N_\alpha$ and we obtain this results concerning the nonlinear boundary conditions.

\begin{theorem}
Let $\lambda \leq 0$ and $ p \geq 2$. Then the Problem \ref{B_param} admits no global positive solution when $\mathcal{B}(u)=0$ stands for the nonlinear boundary conditions $\partial_\nu u = g(u)$, where $g$ is a function such that there exists $c>0$ and $d>0$ satisfying
\begin{displaymath}
g(\eta)  \geq c \eta^2 +	 d \eta .
\end{displaymath}  
\end{theorem}
\begin{proof}
As in the case of $\Omega= (0,\infty)$, we use a weighted $L^1-$norm:
\begin{displaymath}
N_\alpha(t) = \int_{-\infty}^0 u(x,t) e^{\alpha x} \ dx \ , \ \textrm{ with } \alpha >0 . 
\end{displaymath}
We compute $N'_\alpha(t)=\int_{-\infty}^0 \partial_t u(x,t) e^{\alpha x} \ dx $, and using the equations of Problem \ref{B_param}, integration by parts leads to
\begin{displaymath}
N'_\alpha(t) = \int_{-\infty}^0 (\alpha^2 u + \alpha u^2 +u^p) e^{\alpha x} \ dx  + \partial_x u(0,t) - \alpha u(0,t) -\frac{\alpha}{2}u^2(0,t) . 
\end{displaymath}
Thanks to $\partial_\nu u(0,t) =\partial_x u(0,t)$ in $(-\infty,0)$, choosing $\alpha = \min \{ 2c, \ d \}$, we obtain
\begin{displaymath}
N'_\alpha(t) \geq  \int_{-\infty}^0 (\alpha^2 u + \alpha u^2 +u^p) e^{\alpha x} \ dx  \geq  \int_{-\infty}^0 u^p e^{\alpha x} \ dx. 
\end{displaymath}
H\"older inequality leads to the differential equation $N'_\alpha(t) \geq \gamma N^p_\alpha(t)$ with $\gamma>0$. Hence $N_\alpha$ and the solution $u$ blow up in finite time.
\end{proof}

\section*{Acknowledgments} 
The author would like to thank Professor Joachim von Below for his valuable discussions and advices.

\end{document}